\documentclass{amsart}

\numberwithin{equation}{section}

\usepackage{amsmath}
\usepackage{amsfonts}
\usepackage{amsthm}
\usepackage{amssymb}
\usepackage{bm}
\usepackage[applemac]{inputenc}
\usepackage{enumerate}
\newtheoremstyle{mystyle}{}{}{\slshape}{2pt}{\scshape}{.}{ }{} 

\newtheorem{thm}{Theorem}[section]
\newtheorem{cor}[thm]{Corollary}

\newtheorem{prop}[thm]{Proposition}
\newtheorem{lemme}[thm]{Lemma}
\newtheorem{fait}[thm]{Fact}

\newtheorem{question}[thm]{Question}

\theoremstyle{definition}
\newtheorem{defi}[thm]{Definition}
\theoremstyle{mystyle}
\newtheorem{ex}[thm]{Exemple}

\theoremstyle{remark}

\newcommand{\monster}{\mathcal U}

\newcommand{\ignore}[1]{}

\DeclareMathOperator{\tp}{tp}
\DeclareMathOperator{\Lstp}{Lstp}

\DeclareMathOperator{\dcl}{dcl}

\DeclareMathOperator{\Aut}{Aut}
\DeclareMathOperator{\Autf}{Autf}

\makeatletter

\def\indsym#1#2{%
 \setbox0=\hbox{$\m@th#1x$}%
 \kern\wd0%
 \hbox to 0pt{\hss$\m@th#1\mid$\hbox to 0pt{$\m@th#1^{#2}$\hss}\hss}%
 \lower.9\ht0\hbox to 0pt{\hss$\m@th#1\smile$\hss}%
 \kern\wd0}
\newcommand{\ind}[1][]{\mathop{\mathpalette\indsym{#1}}}

\def\nindsym#1#2{%
 \setbox0=\hbox{$\m@th#1x$}%
 \kern\wd0%
 \hbox to 0pt{\hss$\m@th#1\not$\kern1.4\wd0\hss}
 \hbox to 0pt{\hss$\m@th#1\mid$\hbox to 0pt{$\m@th#1^{#2}$\hss}\hss}%
 \lower.9\ht0\hbox to 0pt{\hss$\m@th#1\smile$\hss}%
 \kern\wd0}

\author{Pierre Simon}

\thanks{Partially supported by ValCoMo (ANR-13-BS01-0006), NSF (grant DMS
1665491), and the Sloan foundation.}

\address{Pierre Simon\\
UC Berkeley\\
Mathematics Department, Evans Hall\\
Berkeley, CA, 94720-3840\\
USA}

\email{pierre.simon@berkeley.edu}

\urladdr{http://www.normalesup.org/\textasciitilde{}simon/}

\title{On amalgamation in NTP$_2$ theories and generically simple generics}

\begin{document}

\begin{abstract}
We prove a couple of results on NTP$_2$ theories. First, we prove an amalgamation statement and deduce from it that the Lascar distance over extension bases is bounded by 2. This improves previous work of Ben Yaacov and Chernikov. We propose a line of investigation of NTP$_2$ theories based on S1 ideals with amalgamation and ask some questions. We then define and study a class of groups with generically simple generics, generalizing NIP groups with generically stable generics.
\end{abstract}
\maketitle
\section*{Introduction}

The class of NTP$_2$ theories contains both simple and NIP theories. It is probably the largest class where forking is sufficiently well behaved to be taken seriously. A couple of important facts are known: over extension bases, forking equals dividing (\cite{CherKapl}) and the non-forking ideal is S1 (\cite{CherBY}). In addition, some theorems on groups generalizing similar results for simple and NIP theories have been proved: Hempel and Onshuus \cite{HempOn} construct definable envelopes for abelian and solvable subgroups; \cite{groups_ntp} studies chain conditions and \cite{PRC} sets the foundations for the theory of definably amenable NTP$_2$ groups. More recently \cite{nipinntp} explores analogues of some NIP-like phenomena.

After a first section of preliminaries, the second section of this paper improves some results from \cite{CherBY}: we give a stronger, more natural, amalgamation theorem and, using an argument from Ita\" i Ben Yaacov, deduce from it that Lascar distance is bounded by 2 over extension bases (answering a question from \cite{CherBY}). We also speculate on a strategy for developing the theory of NTP$_2$. We observe that in simple theories, one usually works with a type $p$ over a small set and consider its non-forking extensions all at once. In NIP however, we prefer to fix some global non-forking (or invariant) extension $\tilde p$ of $p$ and study it, possibly using compactness of the space of non-forking extensions at the end of our construction to obtain a result on $p$ itself. Our idea is that in NTP$_2$, one would have to do a mixture of those two things and we suggest that the class of non-forking S1 ideals with amalgamation will replace the ideal of non-forking extensions in simple theories and the use of global invariant types in NIP.

We then turn our attention to definable groups. In the NIP setting, groups with an invariant measure play an important role. Some results concerning them have been generalized to NTP$_2$ in \cite{PRC}. In this paper, we pursue this enterprise by generalizing a subclass: that of groups with a generically stable generic. Such groups play an important role in Hrushovski's study of metastable groups \cite{metastableGr}. Transported to NTP$_2$, the condition becomes that of having a \emph{generically simple generic}. We show that under this assumption all forking-generic types are generically simple and non-forking over any extension base. We leave open the questions of generalizing the classes of fsg groups and compactly dominated groups. There are some natural candidates for these, but we were not able to prove convincing statements about them.

\subsection*{Acknowledgments}

We would like to thank Ita\"i Ben Yaacov for proving and allowing us to include here Theorem \ref{th_Lascar} on Lascar distance. Thanks also to Itay Kaplan for sharing some notes he had written, which helped in writing this paper.

\section{Preliminaries}

Our notations are standard. We work in a complete theory $T$ which has a monster model $\monster$. Usually, $A,B,C,\ldots$ will denote small subsets of $\monster$ and $M,N,\ldots$, small submodels. The group of automorphisms of $\monster$ fixing $A$ pointwise is denoted by $\Aut(\monster/A)$. We will often assume that $T$ is NTP$_2$. For definitions and basic facts about this condition see \cite{CherKapl}. We will actually never use the definition of NTP$_2$, but only certain properties that we recall here.

A subset $A$ of the monster model is an \emph{extension base} if no type $p\in S(A)$ forks over $A$. It is proved in \cite{CherKapl} that if $A$ is an extension base in an NTP$_2$ theory, then forking and dividing over $A$ coincide.

We use the notation $a\ind_C b$ to mean that $\tp(a/Cb)$ does not fork over $C$. We know from \cite{CherKapl} that if $C$ is an extension base, then this relation satisfies extension on both sides: If $a\ind_C b$ (resp. $a\ind_C b$) and $d$ is any tuple, then there is $d'\equiv_{Ca} d$ such that $d'\ind_C b$ (resp. $b\ind_C d'$). Also, in any theory, non-forking satisfies transitivity: if $a\ind_{Cd} b$ and $d\ind_C b$, then $ad\ind_C b$, as well as base monotonicity: if $a\ind_C bd$, then $a\ind_{Cd} b$.

A \emph{Morley sequence} over $A$ is a sequence $(a_i:i<\omega)$ which is $A$-indiscernible and such that $a_i \ind_A a_{<i}$ for all $i<\omega$. 
Recall that Lascar equivalence over a set $A$ is defined as the finest bounded $A$-invariant equivalence relation. A class of this equivalence relation is called a Lascar strong type. We let $\Lstp(a/A)$ denote the Lascar strong type of $a$ over $A$ and $\equiv_A^L$ denote equality of Lascar strong types over $A$. If $a,b$ have the same Lascar strong type over $A$, then there are $a=a_0,a_1,\ldots,a_n=b$ such that $(a_i,a_{i+1})$ starts an $A$-indiscernible sequence for all $i$. The minimal such $n$ is called the Lascar distance of $a$ and $b$ over $A$ and denoted $d_A(a,b)$.

\begin{fait}[\cite{CherBY},Theorem 3.3]\label{f_amal}
Let $T$ be NTP$_2$ and let $A$ be an extension base. Assume that $c\ind_A ab$, $a\ind_A bb'$ and $b\equiv^ L_A b'$. Then there is $c'$ such that $c' \ind_A ab'$, $c'a \equiv_A ca$ $c'b' \equiv_A cb$.
\end{fait}

%
The following lemma will be useful in the next section.

\begin{lemme}\label{lem_lascarext}
Let $T$ be NTP$_2$ and let $A$ be an extension base. Let $a \ind_A b$ and $c\in \monster$. Then there is $ac'\equiv_A^L ac$ such that $ac'\ind_A b$.
\end{lemme}
\begin{proof}
Let $C$ be a set of tuples of same size as $c$ such that for any $c'\in C$, $ac'\equiv_A ac$ and the set $\{\Lstp(ac'/A) : c'\in C\}$ is maximal, that is contains all possible Lascar strong types over $A$ of some $ac'$, with $ac'\equiv_A ac$. By left-extension, there is $C'\equiv_{Aa} C$ such that $aC' \ind_A b$. Then $\{\Lstp(ac'/A) : c'\in C'\}$ is also maximal, hence we can find the $c'$ we are looking for in $C'$.
\end{proof}

\subsection{Measures and ideals}

A (Keisler) measure $\mu(x)$ on a definable set $D$ over $A$ is a finitely additive probability measure on $A$-definable subsets of $D$ (in the variable $x$). Measures play an important role in NIP theories, through invariant measures on groups, but also for the general theory (for instance distality can be defined via properties of measures). In NTP$_2$ theories, we believe a similar role will be played by ideals.

By an ideal, we always mean an ideal on the boolean algebra of definable sets over some $A$. If $I$ is such an ideal, we say that a type $p$ is $I$-wide if it does not imply a formula in $I$.

A measure $\mu(x)$ over $\monster$ is $M$-invariant if $\mu(\phi(x;a))$ depends only on $\tp(a/M)$, equivalently $\mu$ is invariant under $\Aut(\monster/M)$. Similarly, an ideal is $M$-invariant if it is invariant under $\Aut(\monster/M)$. The following definition comes from \cite{hr_appr}.

\begin{defi}
Let $I$ be an $M$-invariant ideal on definable sets. We say that $I$ is S1 if given an $M$-indiscernible sequence $(a_i:i<\omega)$ and a formula $\phi(x;y)$, if $\phi(x;a_0)$ is $I$-wide, then so is $\phi(x;a_0)\wedge \phi(x;a_1)$.
\end{defi}

Note that one then has the stronger property that the full partial type $\{\phi(x;a_i):i<\omega\}$ is $I$-wide. It follows that if $I$ is $A$-invariant and S1, then $I$ contains all formulas which divide over $A$.

If $\mu$ is an $M$-invariant measure, then we can associate to it an ideal $I_\mu$ on definable sets by defining $X\in I_\mu$ if $\mu(X)=0$. Then $I_\mu$ is S1 (see for example \cite{hr_appr} or \cite[Lemma 7.5]{NIPbook}). Another source of S1 ideals comes from the following important fact.

\begin{fait}[\cite{CherBY}, Theorem 2.9]\label{fact_s1}
Let $T$ be NTP$_2$ and let $A$ be an extension base. Then the ideal of formulas that fork over $A$ is ($A$-invariant and) S1.
\end{fait}

\section{Amalgamation}

Throughout this section, we assume that $T$ is NTP$_2$. We will improve Fact \ref{f_amal}. We first show that we can always amalgamate a type with itself.

\begin{prop}
Let $A$ be an extension base. Let $\phi(x;a)$ be non-forking over $A$ and assume that $b\ind_A a$ with $b\equiv_A^ L a$. Then $\phi(x;a)\wedge \phi(x;b)$ is non-forking over $A$.
\end{prop}
\begin{proof}
Let $b\ind_A a$, with $b \equiv_A^ L a$. Build an indiscernible sequence $(b_i:i<\omega)$ in $\tp(b/Aa)$ which is Morley over $Aa$ and with $b_0 = b$. We then have $b_1 \ind_A ba$, $b \equiv_A^ L a$ and by the S1 property, $\phi(x;b)\wedge \phi(x;b_1)$ is non-forking over $A$. By Fact \ref{f_amal}, $\phi(x;a)\wedge \phi(x;b_1)$ is non-forking over $A$. As $\tp(b_1/Aa)=\tp(b/Aa)$, also $\phi(x;a)\wedge \phi(x;b)$ is non-forking over $A$.
\end{proof}

We deduce the following strengthening of Fact \ref{f_amal}.

\begin{thm}\label{th_amalg}
Let $A$ be an extension base and $\phi(x;y),\psi(x;y)$ over $A$. Let $a,b,b'\in \monster$, $b\equiv_A^ L b'$ and either $a\ind_A b'$ or $b'\ind_A a$. Assume that $\phi(x;a)\wedge \psi(x;b)$ is non-forking over $A$, then so is $\phi(x;a)\wedge \psi(x;b')$.
\end{thm}
\begin{proof}
Assume first that $b'\ind_A a$. Let $M\supseteq A$ be a model such that $b'\ind_A Ma$. Find $b''\equiv_{Ma} b'$ such that $b''\ind_A Mab$. Then we have $b'' \equiv_{Aa}^L b'$ and replacing $b'$ by $b''$, we may assume that $b'\ind_A ab$. By Lemma \ref{lem_lascarext}, we can then find $a'$ such that $a'b'\equiv_A^L ab$ and $a'b'\ind_A ab$. By the previous proposition, $(\phi(x;a)\wedge \psi(x;b))\wedge (\phi(x;a')\wedge \psi(x;b'))$ is non-forking over $A$ and a fortiori so is $\phi(x;a)\wedge \psi(x;b')$.

Now assume that $a\ind_A b'$. Let $\sigma$ be a Lascar-strong automorphism over $A$ sending $b$ to $b'$ and set $a'=\sigma(a)$. Then $a\equiv_A ^L a'$, $\phi(x;a')\wedge \psi(x;b')$ does not fork over $A$ and $a\ind_A b'$. We can then apply the previous paragraph to conclude that $\phi(x;a)\wedge \psi(x;b')$ is non-forking over $A$.
\end{proof}

Here is another way to state the result of the theorem.

\begin{cor}\label{cor_amal}
Let $A$ be an extension base. Let $a,b,b',c\in \monster$, $b\equiv_A^ L b'$ and either $a\ind_A b'$ or $b'\ind_A a$. If $c\ind_{A} ab$, then there is $c'$ such that $c'\ind_A ab'$, $c'a\equiv_A ca$ and $c'b'\equiv_A cb$.
\end{cor}
\begin{proof}
Write $p(x;a)=\tp(c/a)$ and $q(x;b)=\tp(c/b)$. By Theorem \ref{th_amalg} the partial type $p(x;a)\wedge q(x;b')$ does not fork over $A$. Take $c'$ to realize a completion of that type over $Aab'$ which is still non-forking over $A$.
\end{proof}

The following consequence of Theorem \ref{th_amalg} is due to Ita\"i Ben Yaacov. It answers Question 3.8 from \cite{CherBY}.

\begin{thm}\label{th_Lascar}
Let $A$ be an extension base and $b\equiv_A ^ L b'$. Then $d_L(b,b')\leq 2$. Furthermore, if $b\ind_A b'$, then $b,b'$ start a Morley sequence over $A$.
\end{thm}
\begin{proof}
It is sufficient to prove the furthermore part, since we can then take $b''\ind_A bb'$ with $b''\equiv_A ^ L b \equiv_A^ L b'$ (by Lemma \ref{lem_lascarext} say) and the sequence $(b,b'',b')$ witnesses that $d_L(b,b')\leq 2$.

Fix some large enough cardinal $\kappa$ and we build by induction a sequence $(b_i:i<\kappa)$ such that for each $i<\kappa$, $b_i\ind_A b_{<i}$ and for $i<j<\kappa$, $b_ib_j\equiv_A bb'$. We will also ensure that the sequence $\tp(b_i/b_{<i})$ is increasing. Start by setting $b_0=b$ and $b_1=b'$. At some limit $\lambda$, let $b_\lambda$ realize $\bigcup_{i<\lambda} \tp(b_i/b_{<i})$. This satisfies the conditions. Assume we have constructed $b_i$ for $i\leq \alpha$ and we look for $b_{\alpha+1}$. Let $p(x;b_{<\alpha})=\tp(b_\alpha/b_{<\alpha})$ and $q(x;b_0)=\tp(b_\alpha/b_0)=\tp(b'/b_0)$. Then $p(x;b_{<\alpha})\cup q(x;b_0)\subseteq \tp(b_\alpha/b_{<\alpha})$ is non-forking over $A$. Since $b_0b_\alpha \equiv_A bb'$, we have $b_\alpha \equiv^ L_A b_0$. Also $b_\alpha \ind_A b_{<\alpha}$ so by Theorem \ref{th_amalg}, the type $p(x;b_{<\alpha})\cup q(x;b_\alpha)$ does not fork over $A$. Take $b_{\alpha+1}$ to realize it so that $b_{\alpha+1}\ind_A b_{\leq \alpha}$. This finishes the construction. Finally, using Erd\H os-Rado, we extract from the sequence $(b_i:i<\kappa)$ an indiscernible subsequence. This gives what we were looking for.
\end{proof}

\subsection{Some speculations and questions}

Let $\Autf(\monster/A)$ be the group of automorphisms if $\monster$ which fix every Lascar-strong type over $A$.

\begin{defi}
Let $A$ be an extension base and let $B,C\subseteq \monster$ contain $A$. Let $p\in S(B)$ and $q\in S(C)$ both non-forking over $A$. We say that $p$ and $q$ are compatible over $A$ if for some/every $\sigma \in \Autf(\monster/A)$ such that either $\sigma(B) \ind_A C$ or $C\ind_A \sigma(B)$, $\sigma(p)(x)\cup q(x)$ is non-forking over $A$.
\end{defi}

We find the following question very appealing.

\begin{question}
Assume that $p$ and $q$ are $M$-invariant types such that $p^{(\omega)}|_M = q^{(\omega)}|_M$. Does it follow that $p$ and $q$ are compatible?
\end{question}

Note that this is true in simple theories because any two $M$-invariant types having the same restriction to $M$ are compatible. It also holds in NIP theories because the condition $p^{(\omega)}|_M = q^{(\omega)}|_M$ implies $p=q$ (\cite[Proposition 2.36]{NIPbook}).



\begin{defi}
Let $A$ be an extension base. An $\Autf(\monster/A)$-invariant S1 ideal $\mu(x)$ has amalgamation if any two $\mu$-wide types are compatible over $A$. 
\end{defi}

\begin{ex}
The dual ideal of a global type non-forking over $A$ has amalgamation over $A$. In an NIP theory those are the only ones since two different non-forking types are never compatible.

In a simple theory, if $p(x)\in \Lstp(A)$, the ideal of formulas $\phi(x)$ such that $p(x)\wedge \phi(x)$ forks over $A$ has amalgamation.
\end{ex}

Let $A$ be an extension base and for simplicity assume Lascar strong types and types over $ A$ coincide. We speculate that $A$-invariant S1 ideals with amalgamation could play in NTP$_2$ theories the same role that $A$-invariant types play in NIP. Of particular importance should be the minimal $A$-invariant S1 ideals with amalgamation. In a simple theory, there is only one such ideal: the ideal of all forking formulas. In NIP, those are precisely duals of $A$-invariant types since two different invariant types can never be amalgamated. In both cases, we see that those minimal ideals partition $A$-invariant types (in two opposite trivial ways). We ask whether this holds in all NTP$_2$ theories.

\begin{question}
Let $A$ be as above. Is the compatibility relation on $A$-invariant types an equivalence relation? Is it the case that if $\mu$ and $\nu$ are two distinct minimal $A$-invariant S1 ideals with amalgamation, then no $A$-invariant type can be wide for both $\mu$ and $\nu$?
\end{question}

\section{Groups}

Recall that a group is definably amenable if it admits a translation-invariant measure, that is a measure $\mu(x)$ on $G$ over some model $M$ with $\mu(g\cdot X)=\mu(X)$ for any $g\in G(M)$ and $M$-definable set $X$. There is a rich theory of NIP definably amenable groups. As shown in \cite{tamedyn}, a group which is not definably amenable cannot admit any notion of \emph{generic type} (where this could mean for example strongly f-generic: no translate forks over some model, or having a small orbit under translation). In NTP$_2$ theories, definable amenability is slightly too strong and the right condition that generalizes jointly simple groups and definably amenable NIP groups is the existence of strongly f-generic types. This was studied in \cite{PRC} and we recall the main results here.

Let $G$ be a group definable in an NTP$_2$ structure $M$.

\begin{defi}\label{strong f generics}
A global type $p\in S_G(\monster)$ is strongly (left) f-generic
over $A$ if for all $g\in G(\monster)$, $g\cdot p$ does not fork
over $A$.

It is strongly bi-f-generic if for all $g,h\in G(\monster)$, $g\cdot p \cdot h$ does not fork over $A$.
\end{defi}

If $G$ admits a strongly f-generic type over some extension base $A$, then it admits a bi-f-generic type over any extension base. When this is the case, we say that $G$ has strong f-generics. Any group definable in a simple theory is such, as is any definably amenable NTP$_2$ group. In NIP, this condition is equivalent to definable amenability.

\begin{defi}
Let $\phi(x)\in L(A)$ be a formula. We say that \emph{$\phi(x)$ is f-generic over $A$} if no (left) translate of $\phi(x)$ forks over $A$. We say that \emph{$\phi(x)$ $G$-divides over $A$} if for some $A$-indiscernible sequence $(g_i:i<\omega)$ of elements of $G$, the partial type $\{g_i\cdot \phi(x):i<\omega\}$ is inconsistent.
\end{defi}

\begin{lemme}[\cite{PRC}, Lemma 3.7]\label{lem_fund1}
Let $A$ be an extension base and $\phi(x)\in L(A)$. Then $\phi(x)$ is f-generic over $A$ if and only if it does not $G$-divide over $A$.
\end{lemme}

Fix some model $M$ and let $\mu_M$ be the ideal of formulas which do not extend to a global type strongly f-generic over $M$. A definable set is \emph{wide} if it does not lie in $\mu_M$ ({\it i.e.}, if it extends to a global type strongly f-generic over $M$). A type is wide if all formulas in it are wide. A type over $M$ is wide precisely if it is f-generic, that is all formulas in it are f-generic.

For any wide type $p$, let $St_{\mu_M}(p) = \{g\in G : gp \cup p$ is wide$\}$. We have the following stabilizer theorem.

\begin{thm}[\cite{PRC}, Theorem 3.18]\label{th_stab}
Assume that $G$ has strong f-generics. Let $p\in S_G(M)$ be f-generic and define $\mu_M$ as above.

Then $G^{00}_M=G^{\infty}_M=St_{\mu_M}(p)^2=(pp^{-1})^2$ and $G^{00}_M\setminus St_{\mu_M}(p)$ is contained in a union of non-$\mu_M$-wide $M$-definable sets.
\end{thm}


\subsection{Generically simple types}

In \cite{CherNTP}, Chernikov defines \emph{simple types} in NTP$_2$ theories (see Definition 6.1 there). We define here a weaker notion of generically simple types. We prove their basic properties following essentially the arguments in \cite{CherNTP}.

\begin{defi}
Let $A$ be any set and $p\in S(A)$. We say that $p$ is generically simple if for every $b\in \monster$ and $a\models p$, $b\ind_{A} a \Longrightarrow a \ind_A b$.
\end{defi}

If $A\subseteq B$, $p\in S(B)$ does not fork over $A$ and $p|_A$ is generically simple, we say that $p$ is generically simple over $A$.

\begin{lemme}\label{lem_gensimpledcl}
Assume that $\tp(a/A)$ is generically simple and $b\in \dcl(Aa)$, then $\tp(b/A)$ is generically simple.
\end{lemme}
\begin{proof}
Let $c\ind_A b$. By taking a non-forking extension of $\tp(c/Ab)$ to $Aa$, we may assume that $c\ind_A a$. Then $a\ind_A c$ and in particular $b\ind_A c$.
\end{proof}

\begin{lemme}\label{lem_gensimpleprod}
If $p,q\in S(A)$ are generically simple, $a\models p$, $a'\models q$ with $a\ind_A a'$, then $\tp(a,a'/A)$ is generically simple.
\end{lemme}
\begin{proof}
Let $b \ind_A aa'$. We have $a \ind_A a'$, hence by transitivity, $ba \ind_A a'$. As $\tp(a'/A)$ is generically simple, $a' \ind_A ba$. On the other hand, we have $b \ind_A a$, hence $a \ind_A b$ by generic simplicity of $\tp(a/A)$. Therefore by transitivity again, $aa'\ind_A b$ as required.
\end{proof}

\begin{lemme}\label{lem_descent}
Assume that $\tp(a/A)$ is generically simple and $\tp(b/Aa)$ is generically simple. Then $\tp(ab/A)$ is generically simple.
\end{lemme}
\begin{proof}
Let $c \ind_A ab$. Then $c\ind_{Aa} b$ and hence $b\ind_{Aa} c$. On the other hand, as $\tp(a/A)$ is generically simple, $a \ind_A c$. By transitivity, $ab\ind_A c$ as required.
\end{proof}

\begin{lemme}\label{lem_symgensimple}
Let $(a_i:i<n)$ be tuples, possibly of different sizes. Assume that for each $i$, $\tp(a_i/A)$ is generically simple and $a_i\ind_A a_{<i}$. Then for any two disjoint subsets $I,J\subseteq n$, we have $a_I \ind_A a_J$ (where $a_I=(a_i:i\in I)$).
\end{lemme}
\begin{proof}
First, by Lemma \ref{lem_gensimpleprod} and induction on $|I|$, $\tp(a_I/A)$ is generically simple for all $I\subseteq n$. Fix $k\leq n$. Then we have $a_{\geq k}\ind_A a_{<k}$ and $a_{<k}\ind_A a_{\geq k}$ by generic simplicity. It follows that $a_{<k}\ind_{Aa_{>k}} a_k$. As also $a_{>k}\ind_A a_k$, we have by transitivity $a_{\neq k}\ind_A a_k$ and $a_k \ind_A a_{\neq k}$ by generic simplicity. This shows that the hypothesis of the lemma is stable under permutation of the indices of the $a_i$'s. The result then follows from the fact that $a_{\leq k}\ind_A a_{>k}$ for all $k<n$.
\end{proof}

Note that the lemma goes through with infinitely many tuples.


\begin{prop}
Let $A$ be an extension base, and assume that $p\in S(A)$ is generically simple. Let $a \ind_A b$ where $a\models p$. Then $b\ind_A a$.
\end{prop}
\begin{proof}
The proof of \cite[Section 6.2]{CherNTP} of the analogue result for simple types goes through using the lemmas above. More precisely, Lemma 6.1 there follows from Fact \ref{fact_s1} (with not assumption of simplicity). In Lemma 6.13, simplicity is only used for checking that property (3) holds. Note that each sequence $\bar b_i$ realizes a generically simple type over the base $A$ by Lemma \ref{lem_gensimpleprod}. We then have by construction $\bar b_i \ind a_{>i}\bar b_{<i}$. So the sequence $(\bar b_0,\ldots,\bar b_i,a_{i+1},a_{i+2},\ldots)$ satisfies the hypothesis of Lemma \ref{lem_symgensimple} and we conclude $a_{>i+1}\bar b_{\leq i} \ind a_{i+1}$ as required. Lemma 6.14 only uses generic simplicity, then Proposition 6.15 goes through unchanged.
\end{proof}

\begin{cor}\label{cor_gensimple}
If $q\in S(B)$ is generically simple over $A$, $A$ an extension base, then it is generically simple itself.

In fact, if $a\models q$ and $b \ind_B a$, then $a \ind_A Bb$.
\end{cor}
\begin{proof}
Let $b$ with $b \ind_B a$. We have $B\ind_A a$ as $\tp(a/A)$ is generically simple and $a\ind_A B$ by assumption. By transitivity, $Bb \ind_A a$, hence $a\ind_A Bb$. In particular, $a\ind_B b$, which shows that $q$ is generically simple.
\end{proof}

\begin{lemme}\label{lem_trans}
Assume that $\tp(a/A)$ is generically simple. Let $A\subseteq B\subseteq C$ with $a\ind_A B$ and $a\ind_B C$. Then $a\ind_A C$.
\end{lemme}
\begin{proof}
As $a\ind_A B$, Corollary \ref{cor_gensimple} implies that $\tp(a/B)$ is generically simple. Therefore we have $C\ind_B a$. Then again by Corollary \ref{cor_gensimple}, $a\ind_A C$.
\end{proof}

The following is the analogue of Problem 6.6 in \cite{CherNTP}.

\begin{question}
Assume that $q\in S(B)$ is generically simple and does not fork over $A$, then is $q|_A$ generically simple?
\end{question}

Note that by Lemma \ref{lem_descent}, this is true if $\tp(B/A)$ is generically simple.

\subsection{Generically simple generics}

We now define a notion of generically simple generic type. We will then prove that if a definable group admits such a type, then all its f-generic types are such and do not fork over any extension base, similarly to what happens with groups in simple theories.

In what follows, $G$ is again a group definable in an NTP$_2$ structure; $S_G(A)$ denotes the space of types over $A$ that concentrate on $G$.

In what follows, we adopt the convention that if say $a,b\in G$, then $ab$ always denotes the product $a\cdot b$ (as opposed to concatenation of tuples, which will be denoted by $a\hat{~}b$).
\begin{defi}
A type $p\in S_G(A)$ is generically simple generic (gsg) if $p$ is generically simple and for any $B\supseteq A$ and $b\in G$, we have $a\ind_A Bb \Longrightarrow b \cdot a\ind_A B b$.

The type $p$ will be said two-sided gsg if it is generically simple and for any $B\supseteq A$ and $b,c\in G$, we have $a\ind_A Bb\hat{~}c \Longrightarrow b \cdot a\cdot c\ind_A Bb\hat{~}c$.
\end{defi}


Note that if $p$ is generically simple generic, then any non-forking extension of $p$ is again generically simple generic. (If say $a\models p$, $a\ind_A B$ and $a\ind_B b$, then $a\ind_A Bb$ by Lemma \ref{lem_trans}.)

\begin{lemme}\label{lem_trangsg}
Let $A$ be an extension base and $\tp(a/A)$ be gsg, $a\ind_A b$, then $\tp(ba/A)$ is generically simple.
\end{lemme}
\begin{proof}
Let $d\ind_A b a$. We can assume furthermore that $d\ind_A ba\hat{~}b\hat{~}a$. Hence $d\ind_{Ab} a$. As $a\ind_A b$, by Corollary \ref{cor_gensimple}, $a\ind_{A} b\hat{~}d$. As $\tp(a/A)$ is gsg, we conclude $ba\ind_A d$ as required.
\end{proof}

It will follow from the statements proved below that in fact $\tp(ba/A)$ is gsg.

\begin{lemme}\label{lem_gsgtranslate}
Let $p\in S_G(A)$ be two-sided gsg and $b,c\in A$. Then $\tp(bac/A)$ is two-sided gsg.
\end{lemme}
\begin{proof}
Assume that $bac \ind_A Bd\hat{~}d'$. Then $a\ind_A Bb\hat{~}c\hat{~}d\hat{~}d'$ since $b,c\in A$. Therefore $dbacd' \ind_A Bb\hat{~}c\hat{~}d\hat{~}d'$ as $\tp(a/A)$ is two-sided gsg.
\end{proof}

\begin{lemme}\label{lem_bigsgexist}
Let $A$ be an extension base, $p\in S_G(A)$ be gsg and take $a,b\models p$, $a\ind_A b$. Then $\tp(ab^{-1}/A)$ is two-sided gsg.
\end{lemme}
\begin{proof}
By Lemma \ref{lem_gensimpleprod}, $\tp(a,b/A)$ is generically simple, therefore so is $\tp(ab^{-1}/A)$. Now let $B\supseteq A$ and $c,d\in G$ such that $ab^{-1} \ind_A Bc\hat{~}d$. By left extension, we may assume that $a\hat{~}b \ind_A Bc\hat{~}d$. Then $Bc\hat{~}d\ind_A a\hat{~}b$ by generic simplicity. Since also $b\ind_A a$, by transitivity, $Bb\hat{~}c\hat{~}d \ind_A a$ and therefore $c a \ind_A Bb\hat{~}c\hat{~}d$ as $\tp(a/A)$ is gsg. Similarly, $d b \ind_A Ba\hat{~}c\hat{~}d$. By transitivity, $ca,db\ind_A Bc\hat{~}d$ and in particular $cab^{-1}d \ind_A Bc\hat{~}d$ as required.
\end{proof}

\begin{prop}\label{prop_gsgkiller}
Let $A\subseteq M$ an extension base. Assume that $p\in S_G(A)$ and $q\in S_G(M)$ are two-sided gsg. Then $q$ does not fork over $A$ and $q|_A$ is (two-sided) gsg.
\end{prop}
\begin{proof}

Let $M\prec N$, $N$ sufficiently saturated. Define the ideal $\mu_M$ as in the introduction to this section. Let $d\in N$ lie in the same $G^{00}_M$-coset as $q$. By Lemma \ref{lem_trangsg}, we can assume that $\tp(d/A)$ is generically simple (take $d$ to be an appropriate translate of a realization of $p$). Let $b\models q$, with $\tp(b/N)$ $\mu_M$-wide and set $b'=db$. Then $\tp(b'/N)$ is $\mu_M$-wide and lies in $G^{00}_M$.

Let $p'$ be a non-forking extension of $p$ to $M$. Then $p'$ is generically simple over $A$ and gsg. By Theorem \ref{th_stab}, the type $b'p'\cup p'$ is $\mu_M$-wide, hence so is $p'\cup b'^{-1}p'$ by left-invariance of $\mu_M$. Let $a$ realize the latter type such that $\tp(a/Nb)$ is $\mu_M$-wide. In particular $a\ind_M Nb$ and $Nb\ind_M a$ as $\tp(a/M)$ is generically simple, and by Corollary \ref{cor_gensimple}, $$(0)\qquad a\ind_A Nb.$$
As $\tp(a/M)$ is gsg over $A$, we deduce
\[(1)\qquad b'a\ind_A Nb.\]
On the other hand, as $\tp(b'/N)$ is two-sided gsg by Lemma \ref{lem_gsgtranslate}, we have 
$$(2)\qquad b'a\ind_N a.$$

By construction, $\tp(b'a/M)=p'$ and thus $\tp(b'a/A)$ is generically simple. By (1), (2) and Lemma \ref{lem_trans}, $b'a\ind_A Na$. Now $\tp(b'a/A)=p$ is two-sided gsg, so we have $b'\ind_A Na$ and $b\ind_A Na$. In particular, $b\ind_A M$. Therefore $q$ does not fork over $A$.


It remains to see that $\tp(b/A)$ is two-sided gsg. We first show that it is generically simple. We know that $\tp(b'a/A)$ is generically simple and $b'a \ind_A Na$, hence $b'a\ind_A d\hat{~}a$. We deduce that $\tp(b'a/Ad\hat{~}a)$ is generically simple, hence so is $\tp(b/Aa)$ by Lemma \ref{lem_gensimpledcl}. Since $a\ind_A N$, $\tp(d\hat{~}a/A)$ is generically simple, by Lemma \ref{lem_descent}, $\tp(b/A)$ is generically simple.

Now assume that $b\ind_A Bc\hat{~}c'$ and we want to show that $cbc'\ind_A Bc\hat{~}c'$. By generic simplicity, $M\ind_A b$. Moving $M$ over $Ab$, we may assume that $M\ind_A Bb\hat{~}c\hat{~}c'$. As $Bc\hat{~}c' \ind_A b$, by transitivity, $MBc\hat{~}c'\ind_A b$. It follows that $b\ind_M Bc$. As $\tp(b/M)$ is gsg, $cbc' \ind_M Bc\hat{~}c'$. Now $\tp(cbc'/Mc\hat{~}c')$ is also two-sided gsg by Lemma \ref{lem_gsgtranslate}, hence by what we have already proved, it is generically simple over $A$. We conclude that $cbc'\ind_A MBc\hat{~}c'$.
\end{proof}

\begin{prop}\label{prop_gsgkiller2}
Let $A\subseteq M$ an extension base. Assume that $p\in S_G(A)$ is two-sided gsg and $q\in S_G(M)$ is f-generic. Then $q$ does not fork over $A$ and $q|_A$ is two-sided gsg.
\end{prop}
\begin{proof}
Let $a\models p$ such that $a\ind_A M$ and let $b\models q$ with $\tp(b/Ma)$ $\mu_M$-wide. Then $ab\ind_M a$ and also $ab\ind_A Mb$. Since $b\ind_M a$ and $\tp(a/M)$ is generically simple over $A$, we have $a\ind_A Mb$. If $N$ is a model containing $Mb$ such that $a \ind_A N$, then by Lemma \ref{lem_gsgtranslate}, $\tp(ab/N)$ is gsg. By Proposition \ref{prop_gsgkiller}, it does not fork over $A$ and is gsg over $A$. Hence $ab\ind_M a$ implies $ab \ind_A Ma$ and then $b\ind_A Ma$ by gsg. In particular, $b\ind_A M$ and $q$ does not fork over $A$. It remains to see that $q|_A$ is two-sided gsg. We know that $\tp(ab/Ma)$ is two-sided gsg, and then so is $\tp(b/Ma)$. By the previous proposition, $\tp(b/A)$ is two-sided gsg.
\end{proof}

\begin{cor}
Assume that $G$ has a gsg type and let $A$ be an extension base. Then any f-generic type of $G$ is non-forking over $A$.
\end{cor}
\begin{proof}
Assume that $G$ has a gsg type over some $B\supseteq A$. Let $p\in S_G(N)$ be an f-generic type. By Proposition \ref{prop_gsgkiller2}, it is gsg over $B$. Let $a\models p$ and take $B'\ind_A aN$, $B'\equiv_A B$. As $\tp(a/N)$ is generically simple over $B$, we have $a\ind_B NB'$. As $B'$ is a conjugate of $B$, there is a gsg type over $B'$ and therefore also a two-sided gsg type. Proposition \ref{prop_gsgkiller2} implies that $a\ind_{B'} N$. As $B'\ind_A N$, we have $a\ind_A N$.
\end{proof}

\begin{question}
Given $G$ an arbitrary definable groups in an NTP$_2$ theory, assume that $p\in S_G(A)$ is generically simple and f-generic, then is it gsg?
\end{question}

\bibliographystyle{alpha}
\bibliography{tout}

\end{document}